\documentclass{article}
\usepackage{amssymb}
\usepackage{amsmath}
\usepackage{subcaption}
\usepackage{float}
\usepackage{graphicx} % Required for inserting images

\newtheorem{Definition}{Definition}[section]
\newtheorem{definition}{Definition}

\newtheorem{theorem}{Theorem}

\newtheorem{remark}{Remark}
\newtheorem{example}{Example}
\newtheorem{proof}{Proof}

\title{Solving and Applying Fractal Differential Equations: Exploring Fractal Calculus in Theory and Practice}
\author{Alireza Khalili Golmankhaneh $^1$\footnote{Corresponding Author }, Donatella Bongiorno $^2$,\\
$^1$ Department of Physics, Urmia Branch, \\Islamic Azad University, Urmia, 63896, Iran\\
alirezakhalili2002@yahoo.co.in\\
$^2$ Department of Engineering, University of Palermo,
Palermo 90100, Italy\\
donatella.bongiorno@unipa.it}

\begin{document}

\maketitle

\begin{abstract}
In this paper, we delve into the fascinating realm of fractal calculus applied to fractal sets and fractal curves. Our study includes an exploration of the method analogues of the separable method and the integrating factor technique for solving $\alpha$-order differential equations. Notably, we extend our analysis to solve Fractal Bernoulli differential equations. The applications of our findings are then showcased through the solutions of problems such as fractal compound interest, the escape velocity of the earth in fractal space and time, and the estimation of time of death incorporating fractal time. Visual representations of our results are also provided to enhance understanding.
\end{abstract}

\textbf{Keywords:} FFractal calculus, Fractal curves,  Fractal differential equations\\
%% PACS codes here, in the form: \PACS code \sep code
\textbf{MSC:} 28A80, 28A78, 28A35, 28A75,  	34A30

\section{Introduction}
Benoit Mandelbrot is credited with pioneering the field of fractal geometry \cite{Mandelbro}, which revolves around shapes possessing fractal dimensions that surpass their topological dimensions \cite{falconer1999techniques,jorgensen2006analysis}. These intricate fractals exhibit self-similarity and frequently demonstrate non-integer and complex dimensions \cite{Qaswet,Lapidus}. However, the analysis of fractals presents challenges, given that traditional geometric measures such as Hausdorff measure \cite{rogers1998hausdorff}, length, surface area, and volume are typically applied to standard shapes \cite{Ewqq}. Consequently, the direct application of these measures to fractal analysis becomes intricate \cite{Barnsley,Gregory,rosenberg2020fractal,Tosatti,bishop2017fractals,Shlomo}.
Researchers have tackled the problem of fractal analysis using approaches. These include analysis \cite{kigami2001analysis,Strichartz2},  measure theory \cite{giona1995fractal,freiberg2002harmonic,jiang1998some,Bongiorno23,bongiorno2018derivatives,bongiorno2015fundamental,bongiorno2015integral}, probabilistic methods \cite{Barlow}, fractional space and nonstandard methods \cite{stillinger1977axiomatic}, fractional calculus \cite{e25071008,uchaikin2013fractional,Trifcebook} and non standard methods \cite{nottale2011scale}.
Essential topological characteristics such as connectivity, ramification, and loopiness were exhibited by fractals and can be quantified using six independent dimension values. However, some fractal types may reduce the count of these dimensions due to their unique traits \cite{patino2023brief}.
Fracture network modeling was proposed, with a specific focus on accentuating fractal attributes within geological formations. Two innovative models were introduced: one centered around Bernoulli percolation within regular lattices, and the other delving into site percolation within scale-free networks integrated into 2D and 3D lattices. The revelation emerged that the effective spatial degrees of freedom in scale-free networks are dictated by the embedding dimension, in contrast to the degree distribution \cite{mondragon2023fractal}. The fractal characteristics impact percolation within self-similar networks \cite{cruz2023percolation}.
The effects of geometric confinement on point statistics in a quasi-low-dimensional system were studied. Specifically, attention was centered on nearest-neighbor statistics. Comprehensive numerical simulations were carried out using binomial point processes on quasi-one-dimensional rectangle strips, considering various confinement ratio values. The findings revealed that the distributions of nearest-neighbor distances followed an extreme value Weibull distribution, where the shape parameter was contingent on the confinement ratio \cite{balankin2023dimensional}.
The fractal characteristics impact formation factors in pore-fracture networks for different transport processes. A focus on deterministic infinitely ramified networks related to pre-fractal Sierpinski carpets was adopted. The network attributes effects on streamline constriction and transmission path tortuosity, emphasizing formation factor differences for diffusibility, electrical conductivity, and hydraulic permeability \cite{balankin2022formation}.
Fractal calculus is a way to extend calculus and deal with equations that have solutions, in the form of functions with fractal properties like fractal sets and curves \cite{parvate2009calculus,parvate2011calculus}. The beauty of fractal calculus lies in its simplicity and algorithmic approaches when compared to methods \cite{Alireza-book}.
The generalization of $F^{\alpha}$-calculus (FC) has been achieved through the utilization of the gauge integral method. The focus lies on the integration of functions within a subset of the real line containing singularities present in fractal sets \cite{golmankhaneh2016fractal}.
The utilization of FC is exemplified with respect to fractal interpolation functions and Weierstrass functions, which can exhibit non-differentiability and non-integrability in the context of ordinary calculus \cite{gowrisankar2021fractal}.
The utilization of non-local fractal derivatives to characterize fractional Brownian motion on thin Cantor-like sets was demonstrated. The proposal of the fractal Hurst exponent establishes its connection to the order of non-local fractal derivatives \cite{golmankhaneh2021fractalBro}.
Various methods have been employed to solve fractal differential equations, and their stability conditions have been determined \cite{golmankhaneh2019sumudu,Fourier1}.
The fractal Tsallis entropy on fractal sets and defines q-fractal calculus for deriving distributions were introduced. Nonlinear coupling conditions for statistical states were presented, and a relationship between fractal dimension and Tsallis entropy's q-parameter in the Hadron system was proposed \cite{golmankhaneh2021tsallis}.
Fractal functional differential equations were introduced as a mathematical framework for phenomena that encompass both fractal time and structure. The paper showcases the solution of fractal retarded, neutral, and renewal delay differential equations with constant coefficients, employing the method of steps and Laplace transforms \cite{golmankhaneh2023initial}.
The introduction of a novel generalized local fractal derivative operator and its exploration in classical systems via Lagrangian and Hamiltonian formalisms were undertaken. The practical applicability of the variational method in describing dissipative dynamical systems was showcased, and the Hamiltonian approach produced auxiliary constraints without reliance on Dirac auxiliary functions \cite{ELNABULSI2022112329}.
Furthermore, fractal stochastic differential equations have been defined, with categorizations for processes like fractional Brownian motion and diffusion occurring within mediums with fractal structures \cite{golmankhaneh2021equilibrium,khalili2019fractalcat,khalili2019random,banchuin2022noise,golmankhaneh2018sub}.
Local vector calculus within fractional-dimensional spaces, on fractals, and in fractal continua was developed. The proposition was put forth that within spaces characterized by non-integer dimensions, it was feasible to define two distinct del-operators-each operating on scalar and vector fields. Employing these del-operators, the foundational vector differential operators and Laplacian in fractional-dimensional space were formulated in a conventional manner. Additionally, Laplacian and vector differential operators linked with $F^{\alpha}$-derivatives on fractals were established \cite{balankin2023vector}.
Fractal calculus has been extended to include Cantor cubes and Cantor tartan \cite{golmankhaneh2018fractalt}, and the Laplace equation has been defined within this framework \cite{khalili2021laplace}.\\
The paper is structured as follows:\\
In Section \ref{1g}, a comparative and review analysis of fractal calculus is presented, focusing on its application to both fractal sets and curves.
Section \ref{2g} introduces the utilization of an integrating factor to solve fractal $\alpha$-order differential equations.
Section \ref{3g} outlines the application of the method of separation to fractal differential equations.
Furthermore, in Section \ref{4g}, various applications are discussed, involving the extension of standard models to account for fractal time.
Lastly, Section \ref{5g} is dedicated to concluding the paper.

\section{Overview of Fractal Calculus\label{1g}}
In this section, we present a comprehensive survey of the application of fractal calculus to the domains of fractal curves and fractal sets \cite{parvate2009calculus,parvate2011calculus,Alireza-book}.
\subsection{Fractal Calculus on Fractal Sets}
In this section, we present a concise overview of fractal calculus applied to fractal sets as summarized in \cite{parvate2009calculus}.
\begin{Definition}
The flag function of a set $F$ and a closed interval $I$ is defined as:
\begin{equation}
  \rho(F,I)=
  \begin{cases}
    1, & \text{if } F\cap I\neq\emptyset;\\
    0, & \text{otherwise}.
  \end{cases}
\end{equation}
\end{Definition}

\begin{Definition}
For a fractal set $F\subset [a,b]$, a subdivision $P_{[a,b]}$ of $[a,b]$, and a given $\delta>0$, the coarse-grained mass of $F\cap [a,b]$ is defined by
\begin{equation}
  \gamma_{\delta}^{\alpha}(F,a,b)=\inf_{|P|\leq
\delta}\sum_{i=0}^{n-1}\Gamma(\alpha+1)(t_{i+1}-t_{i})^{\alpha}
\rho(F,[t_{i},t_{i+1}]),
\end{equation}
where $|P|=\max_{0\leq i\leq n-1}(t_{i+1}-t_{i})$, and $0< \alpha\leq1$.
\end{Definition}

\begin{Definition}
The mass function of a fractal set $F\subset [a,b]$ is defined as the limit of the coarse-grained mass as $\delta$ approaches zero:
\begin{equation}
  \gamma^{\alpha}(F,a,b)=\lim_{\delta\rightarrow0}\gamma_{\delta}^{\alpha}(F,a,b).
\end{equation}
\end{Definition}

\begin{Definition}
For a fractal  set $F\subset [a,b]$, the $\gamma$-dimension of $F\cap [a,b]$ is defined as:
\begin{align}
  \dim_{\gamma}(F\cap
[a,b])&=\inf\{\alpha:\gamma^{\alpha}(F,a,b)=0\}\nonumber\\&
=\sup\{\alpha:\gamma^{\alpha}(F,a,b)=\infty\}
\end{align}
\end{Definition}

\begin{Definition}
The integral staircase function of order $\alpha$ for a fractal set $F$ is given by:
\begin{equation}
 S_{F}^{\alpha}(x)=
 \begin{cases}
   \gamma^{\alpha}(F,a_{0},x), & \text{if } x\geq a_{0}; \\
   - \gamma^{\alpha}(F,x,a_{0}), & \text{otherwise}.
 \end{cases}
\end{equation}
where $a_{0}$ is an arbitrary fixed real number.
\end{Definition}

\begin{Definition}
Let $F$ be an $\alpha$-perfect fractal set, let $f$ be a function defined on F and let $x\in F.$ The $F^{\alpha}$-derivative of $f$ at the point $x$ is defined as follows:
\begin{equation}
  D_{F}^{\alpha}f(x)=
  \begin{cases}
    \underset{ y\rightarrow
x}{F_{-}\text{lim}}~\frac{f(y)-f(x)}{S_{F}^{\alpha}(y)-S_{F}^{\alpha}(x)}, & \text{if } x\in F; \\
    0, & \text{otherwise}.
  \end{cases}
\end{equation}
if the fractal limit $F_{-}\text{lim}$ exists \cite{parvate2009calculus}.
\end{Definition}

\begin{Definition}
Let $I=[a,b]$.~Let $F$ be an $\alpha$-perfect fractal set such that $S^{\alpha}_F$ is finite on $I$. Let $f$ be a bounded function defined on F and let $x\in F.$ The $F^{\alpha}$-integral of $f$ on $I$ is defined as:
\begin{align}
  \int_{a}^{b}f(x)d_{F}^{\alpha}x&=\sup_{P_{[a,b]}}
\sum_{i=0}^{n-1}\inf_{x\in F\cap
I}f(x)(S_{F}^{\alpha}(x_{i+1})-S_{F}^{\alpha}(x_{i}))
\nonumber\\&=\inf_{P_{[a,b]}}
\sum_{i=0}^{n-1}\sup_{x\in F\cap
I}f(x)(S_{F}^{\alpha}(x_{i+1})-S_{F}^{\alpha}(x_{i})).
\end{align}

\end{Definition}

\subsection{Fractal Calculus on Fractal Curves}
We begin with defining the key concepts in fractal calculus on fractal curves \cite{parvate2011calculus}. By the way, we recall that a fractal curve $F\subset \mathbb{R}^n$ is  parametrizable if there exists a bijective and continuous function $\mathbf{w}:[a_0, b_0]\rightarrow\mathbb{R}$. Moreover we recall also that by $C(a,b)$ we denote the segment of the curve lying between the points $\mathbf{w}(a)$ and $\mathbf{w}(b)$ on the fractal curve $F$ \cite{parvate2011calculus}.

\begin{Definition}
For a fractal curve denoted as $F$ and a subdivision denoted as $P_{[a,b]}$ where $[a,b] \subset \mathbb{R}$, the mass function is given by
\begin{equation}
\gamma^{\alpha}(F,a,b)=\lim_{\delta\rightarrow0} \inf_{|P|\leq
\delta}\sum_{i=0}^{n-1}
\frac{|\mathbf{w}(t_{i+1})-\mathbf{w}(t_{i})|^{\alpha}}{\Gamma(\alpha+1)},
\end{equation}
where $|\cdot|$ represents the Euclidean norm in $\mathbb{R}^{n}$, $1\leq
\alpha\leq n$, $P_{[a,b]}=\{a=t_{0},...,t_{n}=b\}$, and
$|P|=\max_{0\leq i\leq n-1}(t_{i+1}-t_{i})$ for a subdivision $P_{[a,b]} $.
\end{Definition}

\begin{Definition}
The $\gamma$-dimension of the fractal curve $F$ is defined as
\begin{align}
\dim_{\gamma}(F) &= \inf\{\alpha:\gamma^{\alpha}(F,a,b)=0\} \nonumber \\
&= \sup\{\alpha:\gamma^{\alpha}(F,a,b)=\infty\}
\end{align}

\end{Definition}

\begin{Definition}
Let $p_0\in[a_0,b_0]$ be arbitrary but fixed. The mass of the fractal of a fractal  curve $F$ is defined as:
\begin{equation}
S_{F}^{\alpha}(u)=\begin{cases}
\gamma^{\alpha}(F,p_{0},u), & u\geq p_{0} ; \\
-\gamma^{\alpha}(F,u,p_{0}), & u<p_{0}.
\end{cases}
\end{equation}
The mass of the fractal curve $F$ up to point $u$ is provided by $S_F^{\alpha}(u)$, where $u\in[a_0,b_0].$
\end{Definition}

\begin{Definition}
 Let $S^\alpha_F(u)=J(\theta).$ The fractal $F^{\alpha}$-derivative of a function $f$ at a point $\theta\in F$ is defined as:
\begin{equation}
D_{F}^{\alpha}f(\theta)=\underset{ \theta'\rightarrow
\theta}{F_{-}lim}~
\frac{f(\theta')-f(\theta)}{J(\theta')-J(\theta)},
\end{equation}
if $F_{-}lim$ exists (here the $F_{-}lim$  represents the fractal limit as it shows in \cite{parvate2011calculus}).
\end{Definition}

\begin{remark}
It is worth noting that the Euclidean distance from the origin to a point
$\theta=\mathbf{w}(u)$ is given by
$L(\theta)=L(\mathbf{w}(u))=|\mathbf{w}(u)|.$
\end{remark}

\begin{Definition}
The fractal integral or $F^{\alpha}$-integral is defined as
\begin{align}
\int_{C(a,b)}f(\theta)d_{F}^{\alpha}\theta &= \sup_{P[a,b]}\sum_{i=0}^{n-1}
\inf_{\theta\in
C(t_{i},t_{i+1})}f(\theta)(J(\theta_{i+1})-J(\theta_{i})) \nonumber \\
&= \inf_{P[a,b]}\sum_{i=0}^{n-1}
\sup_{\theta\in
C(t_{i},t_{i+1})}f(\theta)(J(\theta_{i+1})-J(\theta_{i})),
\end{align}
where $t_{i}=\mathbf{w}^{-1}(\theta_{i})$ and $f$ is a bounded function on a fractal curve $F$.
\end{Definition}

\section{Solving Fractal Differential Equations by Method of Integrating Factor \label{2g}}
In this section, we delve into the concept of differential equations on fractal curves and  fractal sets. We start by considering an $\alpha$-order linear differential equation on a fractal curve $F\subset \mathbb{R}^{n}$:
\begin{equation}\label{eq:fractal_diff_eq}
D_{F}^{\alpha}y(\theta) + p(\theta)y(\theta) = g(\theta),~~~\theta\in F,
\end{equation}
where $p$ and $g$ are $F$-continuous functions defined on the fractal curve $F$, with $\varphi_{1}<\theta<\varphi_{2}$, and $\varphi_{1}, \varphi_{2}\in F$.

\begin{definition}
Let $\psi : F\subset \mathbb{R}^n \rightarrow \mathbb{R}$ be a function. If $\psi$ has fractal $F^{\alpha}$-derivative at each point $\theta\in F,$ therefore $\psi$ is called the solution of the $\alpha$-order differential equation if substituted in the Eq.\eqref{eq:fractal_diff_eq} satisfies it.
\end{definition}

\begin{theorem} (Method of the integration factor)

Let $F\subset \mathbb{R}^n$ be a fractal curve,  therefore there exists a fractal $F^{\alpha}$-differentiable function defined on a fractal curve $F,$ called integration factor,  such that  all the solutions of the Eq. Eq.\eqref{eq:fractal_diff_eq} are expressed by:

\begin{equation}\label{iuokmju}
  y(\theta)=\frac{\int\mu(\theta)g(\theta)d^{\alpha}_F(\theta)\,+\,J(c)}{\mu(\theta)}.
\end{equation}
Here $\mu(\theta)$ is the integration factor and $J(c)$ is an arbitrary constant.

\end{theorem}

\begin{proof}
  To solve Eq. \eqref{eq:fractal_diff_eq}, we introduce an integrating factor $\mu(\theta)$ and multiply both sides of the equation by it:
\begin{equation}\label{eq:modified_fractal_diff_eq}
\mu(\theta) D_{F}^{\alpha}y(\theta) + \mu(\theta)p(\theta)y(\theta) = \mu(\theta)g(\theta).
\end{equation}
For this modified equation to hold, we require the following relationship:
\begin{equation}\label{eq:integrating_factor_condition}
D_{F}^{\alpha}\mu(\theta) = p(\theta) \mu(\theta).
\end{equation}
Assuming $\mu(\theta) > 0$, we can express Eq. \eqref{eq:integrating_factor_condition} as:
\begin{equation}\label{eq:integrating_factor_condition_positive}
\frac{D_{F}^{\alpha}\mu(\theta)}{\mu(\theta)} = p(\theta).
\end{equation}
By applying fractal integration, we arrive at the integral equation:
\begin{equation}\label{eq:integrating_factor_integral}
\ln(\mu(\theta)) = \int p(\theta)  d_{F}^{\alpha}\theta + J(k),
\end{equation}
where $J(k)$ is an arbitrary constant of integration. Setting $J(k) = 0$, we obtain the expression for the integrating factor:
\begin{equation}\label{eq:integrating_factor_expression}
\mu(\theta) = \exp\left(\int p(\theta)  d_{F}^{\alpha}\theta\right).
\end{equation}
After determining $\mu(\theta)$, we substitute it back into Eq. \eqref{eq:modified_fractal_diff_eq}, yielding:
\begin{equation}\label{eq:final_modified_diff_eq}
D_{F}^{\alpha}(\mu(\theta) y(\theta)) = \mu(\theta) g(\theta).
\end{equation}
Integrating both sides of the equation using fractal integration, we arrive at the solution for the original differential equation \eqref{eq:fractal_diff_eq}:
\begin{equation}\label{eq:fractal_diff_eq_solution}
y(\theta) = \frac{\int \mu(\theta) g(\theta)  d_{F}^{\alpha}\theta + J(c)}{\mu(\theta)},
\end{equation}
which completes the proof.
\end{proof}
\begin{remark}
By the previous theorem it follows that there are infinitely many functions $y(\theta)$ that satisfy the given $\alpha$-order fractal
differential equation on the fractal curve $F$.
\end{remark}

\begin{example}
Consider an $\alpha$-order fractal differential equation on fractal curve $F$ of the form:
\begin{equation}\label{eq:example_fractal_diff_eq}
D_{F}^{\alpha}y(\theta) = ry(\theta) + k,
\end{equation}
where $r$ and $k$ are constants. By  Equation \eqref{eq:fractal_diff_eq_solution}, we can determine the infinite solutions for the given equations \eqref{eq:example_fractal_diff_eq} as follows:
\begin{equation}\label{eq:example_solution}
y(\theta) = -\frac{k}{r} + c \exp(rJ(\theta)),
\end{equation}
where $c$ is an integration constant, and $J(\theta)$ arises from the fractal integration process.
\end{example}
\begin{theorem}
Let $p(\theta)$ and $g(\theta)$ two $F$-continuous functions defined on a fractal curve $F$ with $\varphi_1 < \theta < \varphi_2$ and $\varphi_1,\varphi_2\in F.$ Let $\theta_0 \in (\varphi_1,\varphi_2)$.  Therefore for each $y_0\in\mathbb{ R}$ there exists
an unique solution $y = \psi (\theta)$, defined at least in a neighborhood of $\theta_0$,  of the following $\alpha$-order fractal differential equation:
 \begin{equation}\label{eq:theorem_fractal_diff_eq}
   D^{\alpha}_Fy(\theta)+p(\theta)y(\theta)=g(\theta),
 \end{equation}
with the initial condition $y(\theta_0)=y_0.$
\end{theorem}
\begin{proof}
Let us denote by $\Lambda(\theta)$ a primite of $p(\theta)$.  The existence of a solution of the given $\alpha$-order fractal differential equation it was already shown by the method of the integration factor already discussed.  Therefore, denoted by $\mu(\theta)= e^{\Lambda(\theta)}$ we arrive at:
\begin{equation}\label{iop}
  y(\theta)=\,e^{-\Lambda(\theta)}\left(\int\ g(\theta)\,\mu(\theta)\,d^{\alpha}_F(\theta)\,+\,c\right),
\end{equation}
where $c$ is a constant of integration.
Now, replacing the initial condition $y(\theta_0)=y_0$ in the previous equation and choosing the primitive $\Lambda(\theta)$ such that $\Lambda(\theta_0)=0, \ ( i.e.\ \Lambda(\theta)=\int_{C(\theta_0,\theta)}\, p(\tau)d^{\alpha}_F\tau)$  we get the following solution:
\begin{equation}\label{iop}
  y(\theta)= \ \,e^{-\Lambda(\theta)}\left(\int _{C(\theta_0,\theta)}\mu(\tau)g(\tau)d^{\alpha}_F\tau\ +\,y_0\right ).
\end{equation}
To prove the uniqueness of the solution, let us suppose by contradiction that there are two different solutions $y_1(\theta)$ and $y_2(\theta)$ of the Eq.\eqref{eq:theorem_fractal_diff_eq} with the same initial condition $y_1(\theta_0)=y_2(\theta_0)=y_0$.\\
Let $z(\theta)=y_1(\theta)-y_2(\theta)$. Therefore, by the linearity of the $\alpha$-order differential equation, substituting $z(\theta)$ into Eq. \eqref{eq:theorem_fractal_diff_eq}, we obtain:
\begin{equation}\label{rwq1}
  D^{\alpha}_Fz(\theta)\,+\,p(\theta)z(\theta)=0.
\end{equation}
 Now, it is trivial to observe that the function $z(\theta)=0,$ identically null, is a solution of $D^{\alpha}_Fz(\theta)\,+\,p(\theta)z(\theta)=0$ with the initial condition $z(\theta_0)=0.$
So, by the conjugacy of $F^{\alpha}$-calculus and the ordinary calculus \cite{parvate2011calculus}, we can conclude that $z(\theta)=0$ is the unique solution of
\begin{equation}\label{opklxzsw}
  D^{\alpha}_Fz(\theta)\,+\,p(\theta)z(\theta)=0,
\end{equation}
with the initial condition $z(\theta_0)=0.$ Therefore, by the definition of $z(\theta)$ we have that $y_1(\theta)=y_2(\theta)$ and this is a contradiction having assumed that Eq.\eqref{eq:theorem_fractal_diff_eq}, with the initial condition $y(\theta_0)=0,$ could admit two different solutions.
\end{proof}
\begin{example}
Consider the fractal differential equation on a fractal set, expressed as
\begin{equation}\label{wwwwq}
D_{t}^{\alpha}y(t)+\frac{1}{2}y(t)=10+5\sin(2S_{F}^{\alpha}(t))
\end{equation}
where the initial condition is given as $y(0)=0$. To solve Eq.\eqref{wwwwq}, we  introduce the integrating factor $\mu(t)=\exp(S_{F}^{\alpha}(t)/2)$. By multiplying Eq.\eqref{wwwwq} with this factor, performing fractal integration, and applying the initial condition, the resulting solution is:
\begin{align}\label{ssa}
y(t)&=20-\frac{40}{17}\cos(2S_{F}^{\alpha}(t))+\frac{10}{17}
\sin(2S_{F}^{\alpha}(t))-\frac{300}{17}
\exp\left(-\frac{S_{F}^{\alpha}(t)}{2}\right)\nonumber\\&
\propto 20-\frac{40}{17}\cos(2t^{\alpha})+\frac{10}{17}
\sin(2t^{\alpha})-\frac{300}{17}
\exp\left(-\frac{t^{\alpha}}{2}\right).
\end{align}
\begin{figure}[H]
  \centering
  \includegraphics[scale=0.5]{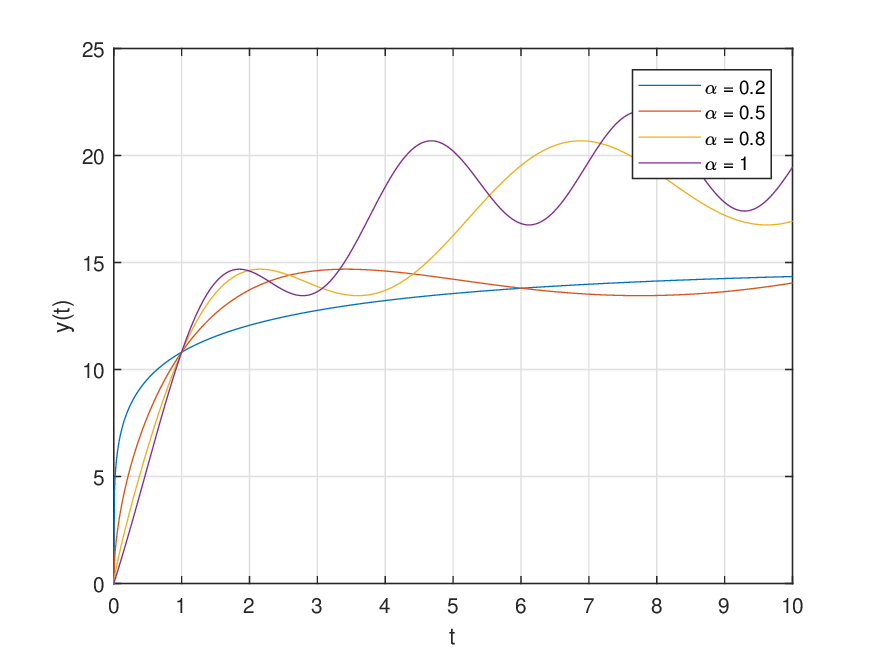}
  \caption{Plot of Eq.\eqref{ssa} for different values of $\alpha$}\label{iiio}
\end{figure}
In Figure \ref{iiio}, we have plotted Eq. \eqref{ssa} for various values of $\alpha$. This plot illustrates that as the dimension of the fractal set supporting the function increases, the solution exhibits greater oscillations.
\end{example}

\subsection{Fractal Bernoulli Differential Equation}
The fractal version of the Bernoulli differential equation, using the fractal derivative, is expressed as:
\begin{equation}\label{uuiuuui}
D_F^\alpha y(\theta) + q(\theta)y(\theta) = r(\theta)y(\theta)^\beta,~~~\theta\in F,~ \beta \in \mathbb{R},
\end{equation}
where $q(\theta)$ and $r(\theta)$ are $F$-continuous functions defined on a fractal curve. This equation employs the fractal derivative to describe the behavior of the function $y$ on a fractal curve, allowing for the analysis and solution of differential equations on fractal geometries.
In order to give a technique to solve the assigned Fractal Bernoulli differential equation, let us observe, first of all, that if $\beta=0$ then $y^{\beta}(\theta)=y^0(\theta)=1$ and the Eq.\eqref{uuiuuui} becomes equal to the Eq.\eqref{eq:fractal_diff_eq} with $r(\theta)=g(\theta).$ Instead, if $\beta=1$, then $y^{\beta}(\theta)=y(\theta)$ and the Eq.\eqref{uuiuuui} becomes similar to the Eq.\eqref{eq:fractal_diff_eq},  here $q(\theta)-r(\theta)=p(\theta).$ In all other cases,  to solve the fractal version of the Bernoulli differential equation, using fractal derivatives,  we apply the integrating factor method already discussed.  Before showing this technique we note that if $\beta>0,$ then $y(\theta)=0$ is a solution of the Eq.\eqref{uuiuuui}.  Now,  the solution method is the following: preliminarily divide both sides of the equation by $y^{\beta}(\theta)$, thus obtaining
\begin{equation}\label{rtea}
  y^{-\beta}(\theta)\left(D^{\alpha}_F y(\theta)\,+\,q(\theta)\,y(\theta)\right)\,=\,r(\theta),
\end{equation}
subsequently set $z(\theta)=y^{1-\beta}(\theta)$ and applying the fractal differentiation rule of composite functions:
$D^{\alpha}_F z(\theta)=(1-\beta)\,y^{-\beta}(\theta)\,D^{\alpha}_F y(\theta),$ so the new Fractal Bernoulli  differential equation is
\begin{equation}\label{op25}
  D^{\alpha}_F z(\theta)\,+\,(1-\beta)\,q(\theta)\,z(\theta)\,=\,(1-\beta)\,r(\theta).
\end{equation}
Therefore apply the integrating factor method and set $y(\theta)=\left(z(\theta)\right)^{\frac{1}{1-\beta}},$ so the solution of the assigned Fractal Bernoulli differential equation is:
\begin{equation}\label{xxqwe}
  y(\theta)=\left(\frac{\int r(\theta)\,\mu(\theta)\,d^{\alpha}_F(\theta)\,+\,c}{\mu(\theta)}\right)^{\frac{1}{1-\beta}},
\end{equation}
where $\Lambda(\theta)$ is a primite of $q(\theta)$ and $\mu(\theta)= e^{\Lambda(\theta)}.$

\begin{example}
 Let us consider the following Fractal Bernoulli differential equation on a given fractal curve $F$:
$$D^{\alpha}_F y(\theta)\,=\,2\,\frac{y(\theta)}{S^{\alpha}_F(\theta)}\,+\,2\,S^{\alpha}_F(\theta)\,\sqrt{y(\theta)}$$ here $\beta=1/2,$ so $y(\theta)=0$  is a solution of the given Fractal Bernoulli differential equation.
Let us suppose, now, that $y(\theta)\neq 0.$
Let us divide both sides of the equation by $\sqrt{y(\theta)}$ and let us set $z(\theta)=\sqrt{y(\theta)}.$ Thus we obtain the following $\alpha$-order linear differential equation  on a given fractal curve $F$:
\begin{equation}\label{rteww}
  D^{\alpha}_F z(\theta)=S^{\alpha}_F(\theta)+\frac{z(\theta)}{S^{\alpha}_F(\theta)}
\end{equation}
According to the formula of the integrating factor method we have:
\begin{equation}\label{rteqa}
  z(\theta)=S^{\alpha}_F(\theta)(S^{\alpha}_F(\theta)\,+\,c).
\end{equation}
Therefore the solutions of the assigned Fractal Bernoulli differential equation are:
$y(\theta)=\left((S^{\alpha}_F(\theta))^2+c\,S^{\alpha}_F(\theta)\right)^2$ and $y(\theta)=0.$
\end{example}
\section{Solving Fractal Differential Equations by Method of Separation \label{3g}}
The equation representing a separable $\alpha$-order fractal differential equation is given as \cite{khalili2023non}:
\begin{equation}\label{yyuhh54p}
D_{F}^{\alpha}y(\theta) = \frac{d_{F}^{\alpha}y}{d_{F}^{\alpha}\theta}= f(\theta,y),~~~\theta\in F,
\end{equation}
with the initial condition
\begin{equation}\label{iopp}
  y(\theta_{0})=y_{0}.
\end{equation}
Here, $f(\theta,y)$ takes a linear form with respect to $y$. The equation \eqref{yyuhh54p} can be rearranged as:
\begin{equation}\label{plmnb}
M(\theta,y)+N(\theta,y)\frac{d_{F}^{\alpha}y}{d_{F}^{\alpha}\theta}=0,
\end{equation}
where $M(\theta,y)=-f(\theta,y)$ and $N(\theta,y)=1$. By considering $M(\theta,y)=M(\theta)$ and $N(\theta,y)=N(y)$, we arrive at:
\begin{equation}\label{okpli}
M(\theta)d_{F}^{\alpha}\theta+N(y)d_{F}^{\alpha}y=0.
\end{equation}
This equation is referred to as a separable fractal differentiable equation. To solve \eqref{okpli}, we introduce functions $D_{F}^{\alpha}H_{1}(\theta)=M(\theta)$ and $D_{F}^{\alpha}H_{2}(y)=N(y)$. Therefore Eq. \eqref{okpli} has the following expression:
\begin{equation}\label{wqqaqq}
D_{F}^{\alpha}H_{1}(\theta)+D_{F}^{\alpha}H_{2}(y)
\frac{d_{F}^{\alpha}y}{d_{F}^{\alpha}\theta}=0.
\end{equation}
Now by fractal chain rule, we have:
\begin{equation}\label{mmnb85}
D_{F}^{\alpha}H_{2}(y)
\frac{d_{F}^{\alpha}y}{d_{F}^{\alpha}\theta}=
\frac{d_{F}^{\alpha}}{d_{F}^{\alpha}\theta}H_{2}(y).
\end{equation}
Consequently, by Eq. \eqref{wqqaqq}  and  Eq. \eqref{mmnb85}  we have:
\begin{equation}\label{wass}
\frac{d_{F}^{\alpha}}{d_{F}^{\alpha}\theta}[H_{1}(\theta)+H_{2}(y)]=0,
\end{equation}
and applying fractal integration, we obtain:
\begin{equation}\label{weedz}
H_{1}(\theta)+H_{2}(y)=c
\end{equation}
The Eq.\eqref{weedz} is the implicit solution of the Eq.\eqref{okpli}.
Now by substituting the initial condition into Eq.\eqref{weedz} we get:
\begin{equation}\label{Nuyt}
c=H_{1}(\theta_{0})+H_{2}(y_{0}).
\end{equation}
Finally by replacing \eqref{Nuyt} into \eqref{weedz}, we arrive at:
\begin{equation}\label{wwwq}
H_{2}(y)-H_{2}(y_{0})=\int_{y_{0}}^{y}N(s)d_{F}^{\alpha}s,~~~
H_{1}(\theta)-H_{1}(\theta_{0})=\int_{C(\theta_0,\theta)}M(s)d_{F}^{\alpha}s
\end{equation}
This leads to:
\begin{equation}\label{aqwerty}
\int_{y_{0}}^{y}N(s)d_{F}^{\alpha}s+
\int_{C(\theta_0,\theta)}M(s)d_{F}^{\alpha}s=0
\end{equation}
which is the implicit solution of \eqref{okpli}, satisfying the initial condition.

\begin{example}
Let's consider the fractal differential equation on a fractal curve  given by
\begin{equation}\label{yyhtuu}
D_{F}^{\alpha}y(\theta)=\frac{3J(\theta)^2+4J(\theta)+2}{2(y-1)},~~~y(0)=-1.
\end{equation}
We can rewrite Eq.\eqref{yyhtuu} as follows:
\begin{equation}
2(y-1)d_{F}^{\alpha}y=(3J(\theta)^2+4J(\theta)+2)d_{F}^{\alpha}\theta.
\end{equation}
Moreover, by fractal integration with respect to $y$ on the left side and with respect to $\theta$ on the right side, we obtain:
\begin{equation}\label{sweq}
y^{2}-2y=J(\theta)^{3}+2J(\theta)^{2}+2J(\theta)+c.
\end{equation}
Finally, by using the initial condition $y(0)=-1$ in  Eq. \eqref{sweq} we get:
\begin{equation}\label{er}
y^{2}-2y=J(\theta)^{3}+2J(\theta)^{2}+2J(\theta)+3.
\end{equation}
It can be further simplified to:
\begin{equation}\label{frtg}
y(\theta)=1-\sqrt{J(\theta)^{3}+2J(\theta)^{2}+2J(\theta)+4}.
\end{equation}
\begin{figure}[H]
  \centering
  \includegraphics[scale=0.5]{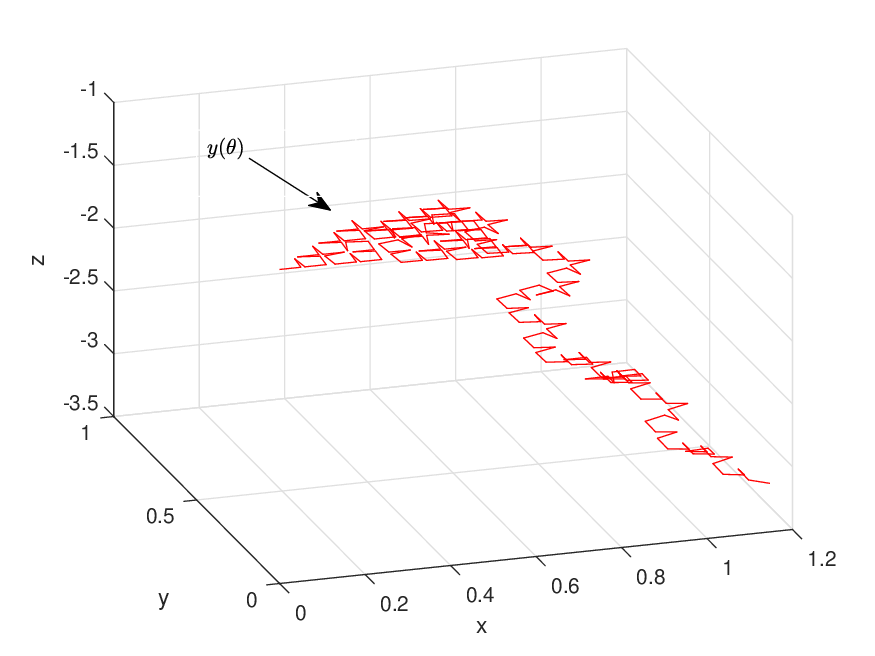}
  \caption{Graph of Eq.\eqref{frtg}.}\label{oo}
\end{figure}
In Figure \ref{oo}, we have depicted the graphical representation of the solution to equation \eqref{yyhtuu}.
\end{example}
\section{Applications \label{4g}}
 In this section, we explore practical applications of fractal differential equations.\\
\subsection{Fractal Compound Interest}
Consider a scenario where a sum of money is deposited in a bank, and both deposits and withdrawals occur at a constant rate $k$ \cite{boyce2021elementary}. The value $p(t)$ of the investment over time represents this situation. The rate of change of $p(t)$ in a fractal time context is given by the equation:
\begin{equation}\label{ewwqas}
D_{F}^{\alpha}p(t)=rp(t)+k,~~~~~~~t\in F,~~~r,k\in \mathbb{R}.
\end{equation}
Here, $r$ represents the annual interest rate and $k$ the constant rate of deposits or withdrawals. The initial condition is $p(0)=p_{0}$. The solution to Eq.\eqref{ewwqas} is derived as:
\begin{align}\label{olp}
p(t)&=p_{0}\exp(rS_{F}^{\alpha}(t))+\frac{k}{r}(\exp(rS_{F}^{\alpha}(t)-1))\\
&\propto p_{0}\exp(rt^{\alpha})+\frac{k}{r}(\exp(rt^{\alpha}-1)).
\end{align}
This solution showcases how the investment's value changes over fractal time, with implications for compound interest calculations.
\begin{figure}[H]
  \centering
  \includegraphics[scale=0.6]{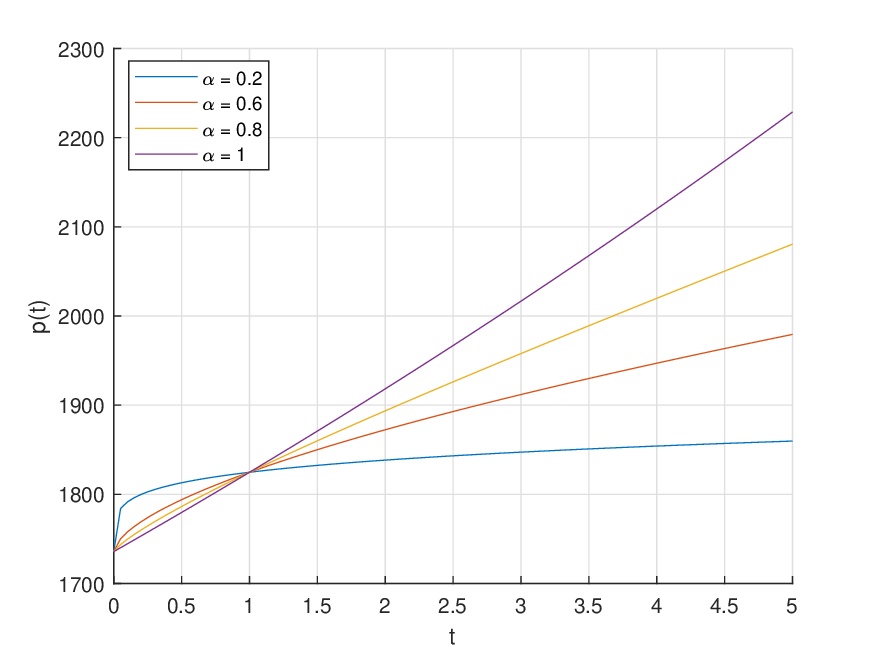}
  \caption{Investment growth for different values of $\alpha$}\label{uyiio}
\end{figure}
In Figure \ref{uyiio}, we illustrate the impact of the fractal time dimension on the growth of investment.

\subsection{Escape Velocity in Fractal Space and Time}
The concept of escape velocity, a fundamental aspect of physics, pertains to the minimum initial velocity required for an object to overcome a celestial body's gravitational pull \cite{boyce2021elementary}. By introducing the concept of fractal space and time, we extend the exploration of escape velocity to an innovative framework.\\
Assuming the absence of other forces and incorporating Newton's law, the equation of motion in fractal space and time can be expressed as:
\begin{equation}\label{eq:escape_velocity_eqn}
mD_{t}^{\alpha} v=-\frac{mgR^2}{(R+x)^2},~~~v(0)=0.
\end{equation}
Here, $m$ represents the mass of the object, $R$ is the radius of the celestial body (such as Earth), $x$ is the distance between the object and the celestial body, and $g$ signifies the acceleration due to gravity. The equation \eqref{eq:escape_velocity_eqn} is then transformed using the fractal chain rule to obtain:
\begin{equation}\label{eq:fractal_velocity_eqn}
vD_{x}^{\alpha} v=-\frac{gR^2}{(R+x)^2}.
\end{equation}
Solving this equation involves separating variables and performing fractal integration, resulting in the equation:
\begin{equation}\label{eq:max_altitude_eqn}
\frac{S_{F}^{\alpha}(v)^2}{2}=\frac{gR^2}{R+S_{F}^{\alpha}(x)}+c.
\end{equation}
Utilizing the initial conditions $x=0$ and $v=v_{0}$, the maximum altitude reached by the object can be determined as:
\begin{equation}\label{eq:max_altitude_sol}
x_{\text{max}}=\frac{S_{F}^{\alpha}(v_{0})^2 R}{2Rg-S_{F}^{\alpha}(v_{0})^2}.
\end{equation}
To find the initial velocity $S_{F}^{\alpha}(v_{0})$ required to elevate the object to the altitude $x_{\text{max}}$, the equation $S_{F}^{\alpha}(v)=0$ is employed, yielding:
\begin{equation}\label{eq:init_velocity_sol}
S_{F}^{\alpha}(v_{0})=\sqrt{2gR\frac{x_{\text{max}}}{R+x_{\text{max}}}}.
\end{equation}
As the concept of escape velocity extends to fractal space, the fractal escape velocity $S_{F}^{\alpha}(v_{e})$ is determined by allowing $x_{\text{max}}$ to approach infinity:
\begin{equation}\label{eq:fractal_escape_velocity}
S_{F}^{\alpha}(v_{e})=\sqrt{2gR},
\end{equation}
or equivalently,
\begin{equation}\label{eq:fractal_velocity_relation}
v_{e}\propto (2gR)^{1/2\alpha}.
\end{equation}
This study delves into the idea of escape velocity within the context of fractal space and time, offering fresh perspectives on the dynamics of entities in dimensions marked by intricacy and self-replication.
\begin{figure}[H]
  \centering
  \includegraphics[scale=0.5]{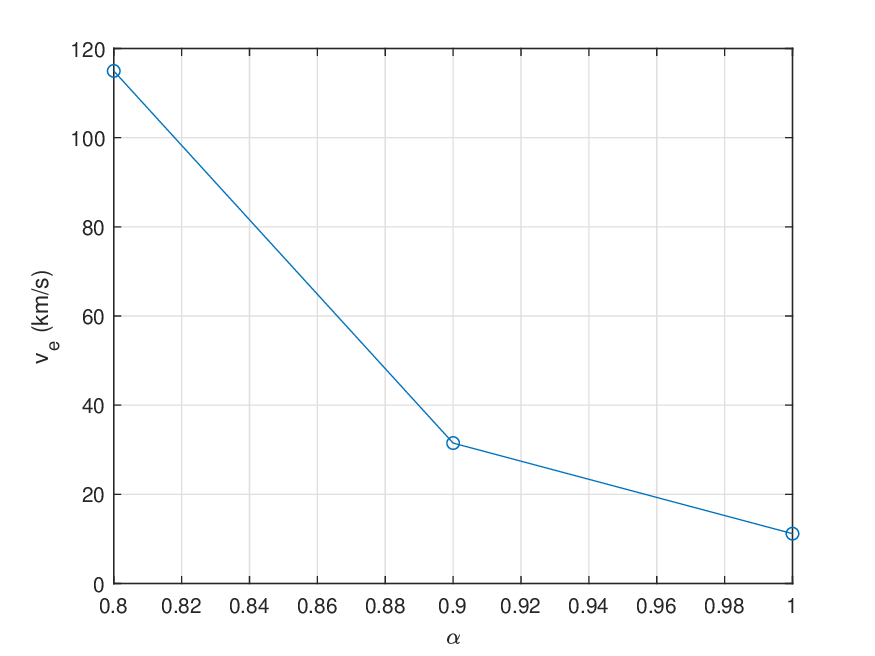}
  \caption{Escape Velocity vs. $\alpha$ }\label{ytuuy}
\end{figure}
As illustrated in Figure \ref{ytuuy}, we observe a pattern where reducing the spatial dimension necessitates an increase in escape velocity.
\subsection{Fractal Newton's Law of Cooling}
Newton's Law of Cooling is a fundamental principle in thermodynamics and heat transfer that describes how the rate of heat transfer between an object and its surroundings changes over time. It's commonly used to model the cooling or heating of an object through conduction, convection, or radiation \cite{molnar1969application,hurley1974application}.\\
The concept of the fractal Newton's Law of Cooling can be presented in the following manner:
\begin{equation}\label{reeeqaqw}
D_{t}^{\alpha}T = -k(T-T_{s}).
\end{equation}
Where:
\begin{align*}
&D_{t}^{\alpha}T ~ \text{signifies the fractal time derivative of temperature,} \\
&T  \text{ is the temperature of the object at any given time,} \\
&T_{s} \text{ is the temperature of the surrounding environment,} \\
& k >0 \text{ is the cooling or heating rate coefficient.}
\end{align*}
In this context, the rate of temperature change with respect to time is described through a fractal time derivative. This approach provides a unique lens through which to comprehend how objects interact with their surroundings, considering the intricacies of fractal time.
\subsection*{Estimation of Time of Death}
As an example, let's estimate the time of death for a body. We assume that the temperature of the body is discovered at time $t=0$ as $T_{0}$, and when it died at $t_{d}$, the temperature was $T_{d}$. By utilizing the cooling law, we can determine $t_{d}$ by solving Eq. \eqref{reeeqaqw}, which can be represented as:
\begin{equation}\label{wqas96}
  T(t) = T_{s} + (T_{0} - T) \exp(-kS_{F}^{\alpha}(t)).
\end{equation}
Here, $T(0) = T_{0}$. If we measure the temperature of the deceased body at time $t = t_{1}$ and find $T = T_{1}$, we can use Eq. \eqref{reeeqaqw} to derive the equation:
\begin{equation}\label{aqwee4}
  T_{1} - T_{s} = (T_{0} - T) \exp(-kS_{F}^{\alpha}(t_{1})).
\end{equation}
From this equation, we can deduce:
\begin{equation}\label{aqxs}
  k = -\frac{1}{S_{F}^{\alpha}(t_{1})}\ln \frac{T_{1} - T_{s}}{T_{0} - T_{s}} \propto -\frac{1}{t_{1}^{\alpha}}\ln \frac{T_{1} - T_{s}}{T_{0} - T_{s}}.
\end{equation}
By substituting $t=t_d$ and $T=T_d$ in the Eq. \eqref{wqas96}, we have:
\begin{equation}\label{e6er}
 S_{F}^{\alpha}(t_{d}) = -\frac{1}{k} \ln \frac{T_{d} - T_{s}}{T_{0} - T_{s}}.
\end{equation}
This equation can also be written as:
\begin{equation}\label{eer}
t_{d} = \bigg|-\frac{1}{k} \ln \frac{T_{d} - T_{s}}{T_{0} - T_{s}}\bigg|^{1/\alpha}.
\end{equation}
This methodology provides a means for estimating the time of death based on temperature measurements and the principles of the fractal Newton's Law of Cooling.
\begin{figure}[H]
  \centering
  \includegraphics[scale=0.5]{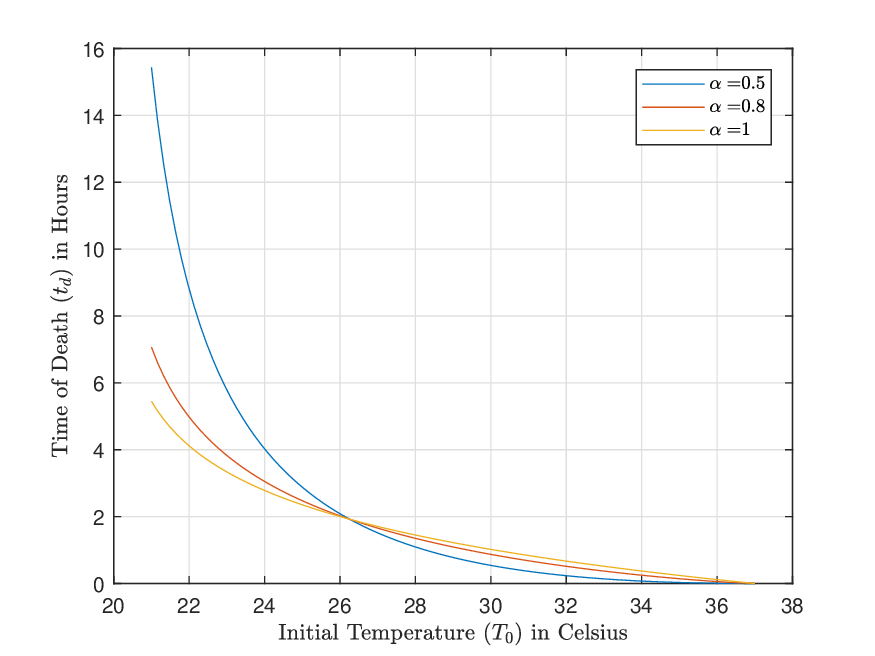}
  \caption{Estimation of Time of Death for Different $\alpha$ Values}\label{reftt}
\end{figure}
Figure \ref{reftt} illustrates that during the initial 2 hours, the cooling rate of the deceased body is faster in the fractal time model compared to the standard time case. However, this trend reverses beyond the 2-hour mark.

\section{Conclusion \label{5g}}
In conclusion, the study delved into a comprehensive exploration of various methods in the realm of fractal calculus. By investigating the method analogues of the separable method and integrating factor technique, we addressed $\alpha$-order differential equations. An intriguing extension of our analysis led to the resolution of Fractal Bernoulli differential equations, further broadening the scope of our inquiry.
The practical implications of our findings were exemplified through their applications in solving real-world problems. From fractal compound interest to the escape velocity of earth in fractal space and time, and even the estimation of time of death with the incorporation of fractal time, our research showcased the versatility of fractal calculus in tackling complex scenarios.
To enhance the accessibility of our work, we provided visual representations of our results. These aids not only conveyed our findings more effectively but also aided in the deeper understanding of the intricate concepts discussed.
A holistic perspective on the applications of fractal calculus has been offered, demonstrating its adaptability and utility in solving diverse problems across various fields.\\
\textbf{Declaration of Competing Interest:}\\
The authors declare that they have no known competing financial interests or personal relationships that could have appeared to influence the work reported in this paper.\\
\textbf{CRediT author statement:}\\
Alireza.K.Golmankhnaeh : Investigation, Methodology, Software, Writing- Original draft preparation.
 Donatella Bongiorno : Investigation, Writing- Reviewing and Editing.\\
\textbf{Declaration of generative AI and AI-assisted technologies in the writing process.}
During the preparation of this work the authors used GPT in order to correct grammar and writing. After using this GPT, the authors reviewed and edited the content as needed and takes full responsibility for the content of the publication.

\section{References}
\bibliographystyle{elsarticle-num}
\bibliography{Refrancesma9}

%% else use the following coding to input the bibitems directly in the
%% TeX file.

%%\begin{thebibliography}{00}

%% \bibitem[Author(year)]{label}
%% For example:

%% \bibitem[Aladro et al.(2015)]{Aladro15} Aladro, R., Martín, S., Riquelme, D., et al. 2015, \aas, 579, A101

%%\end{thebibliography}

\end{document}